\documentclass[a4paper, reqno]{amsart}

\usepackage{
	 amsmath
	,amssymb
	,amsfonts
	,amsthm
	,xcolor
	,todonotes
	,enumitem
	,graphicx
	,tikz
	,xspace
	,mathtools
	,xparse
	,ulem
	,stmaryrd}

\usepackage{tikz-cd}
\renewcommand{\textbf}[1]{\text{\fontseries{b}\selectfont{\upshape #1}}}
\newcommand{\xto}[1]{\xrightarrow{#1}}

\def\C{\mathcal C}

\theoremstyle{definition}
	\newtheorem{thm}{Theorem}[section]
	\newtheorem{lem}[thm]{Lemma}

	\newtheorem{ex}[thm]{Exercise}
	\newtheorem{oss}[thm]{Remark}
	\newtheorem{df}[thm]{Definition}
	
	\newtheorem{notat}[thm]{Notation}
    
	\newtheorem*{thm*}{Theorem}
	\newtheorem*{lem*}{Lemma}
	\newtheorem*{prop*}{Proposition}
	\newtheorem*{es*}{Example}
	\newtheorem*{ex*}{Exercise}
	\newtheorem*{oss*}{Remark}
	\newtheorem*{df*}{Definition}
	\newtheorem*{cor*}{Corollary}
	\newtheorem*{notat*}{Notation}
    \newtheorem*{conjec*}{Conjecture}

\usepackage[all,cmtip]{xy}
\usepackage[bottom=6cm,
      left=4cm,
      right=4cm,
      top=4.5cm]{geometry}
\usepackage{hyperref}
\hypersetup{%
  colorlinks=true,
  linkcolor=black,
  urlcolor=blue!70,
  citecolor=red}

\newlength{\seplen}
\newlength{\sepwid}
\def\firstblank{\,\rule{\seplen}{\sepwid}\,}
\def\secondblank{\firstblank\llap{\raisebox{2pt}{\firstblank}}}

\DeclareDocumentCommand{\lbkt}{ O{\firstblank} O{\secondblank} }{\langle #1 , #2 |}
\DeclareDocumentCommand{\rbkt}{ O{\firstblank} O{\secondblank} }{| #1 , #2\rangle}

\DeclareDocumentCommand{\llbkt}{ O{\firstblank} O{\secondblank} }{\langle\kern-.2em\langle #1 , #2 |\kern-.125em|}
\DeclareDocumentCommand{\rrbkt}{ O{\firstblank} O{\secondblank} }{|\kern-.125em| #1 , #2\rangle\kern-.2em\rangle}

\renewcommand{\diamond}{\odot}
\renewcommand{\circ}{\cdot}
\newcommand{\tens}{\circledcirc}

\newcommand{\red}[1]{\textcolor{blue!80!white}{(\text{\oldstylenums{#1}})}}

\ExplSyntaxOn
\NewDocumentCommand{\makeabbrev}{mmm}
 {
  \yoruk_makeabbrev:nnn { #1 } { #2 } { #3 }
 }

\cs_new_protected:Npn \yoruk_makeabbrev:nnn #1 #2 #3
 {
  \clist_map_inline:nn { #3 }
   {
    \cs_new_protected:cpn { #2 } { #1 { ##1 } }
   }
 }
\ExplSyntaxOff

\makeabbrev{\mathbf}{b#1}{
    b,c,d,e,g,h,i,j,k,l,m,n,o,p,q,r,t,u,v,w,x,y,z,%
    A,B,C,D,E,G,H,I,J,K,L,M,N,O,P,Q,R,S,T,U,V,W,X,Y,Z}

\makeabbrev{\boldsymbol}{bs#1}{%
    a,b,c,d,e,g,h,i,j,k,l,m,n,o,p,q,r,s,t,u,v,w,x,y,z,%
    A,B,C,D,E,G,H,I,J,K,L,M,N,O,P,Q,R,S,T,U,V,W,X,Y,Z}

\makeabbrev{\mathsf}{sf#1}{a,b,c,d,e,f,g,h,i,j,k,l,m,n,o,p,q,r,s,t,u,v,w,x,y,z,%
                           A,B,C,D,E,F,G,H,I,J,K,L,M,N,O,P,Q,R,S,T,U,V,W,X,Y,Z}

\makeabbrev{\mathfrak}{f#1}{a,b,c,d,e,f,g,h,j,k,l,m,n,o,p,q,r,s,t,u,v,w,x,y,z,%
                             A,B,C,D,E,F,G,H,I,J,K,L,M,N,O,P,Q,R,S,T,U,V,W,X,Y,Z}

\makeabbrev{\mathcal}{cl#1}{A,B,C,D,E,F,G,H,I,J,K,L,M,N,O,P,Q,R,S,T,U,V,W,X,Y,Z}
\makeabbrev{\mathbf}{l#1}{A,B,C,D,E,F,G,H,I,J,K,L,M,N,O,P,Q,R,S,T,U,V,W,X,Y,Z}
\makeabbrev{\mathbb}{s#1}{A,B,C,D,E,F,G,H,I,J,K,L,M,N,O,P,Q,R,S,T,U,V,W,X,Y,Z}

\def\thc{\textsc{thc}\@\xspace}
\def\C{\textsf{C}}

\def\Vcat{\mathcal{V}\text{-}\textbf{Cat}}
\def\B{\sfB}
\def\A{\sfA}
\def\S{\sfS}
\def\t{\clT}

\title[a standard theorem on thc]{A standard theorem on\\ adjunctions in two variables}

\author{Fosco Loregian}
\usepackage{turnstile}
  \newcommand{\adjunct}[2]{\nsststile{#2}{#1}}

\def\[{\begin{equation}}
\def\]{\end{equation}}

\def\eqref#1{(\ref{#1})}

\makeatletter
\newcommand{\centeredxymatrix}[1]{\vcenter{\xymatrix{#1}}}
\makeatother

\address{
	Fosco \textsc{Loregian}\newline
	Max Planck Institute for Mathematics\newline
	Vivatsgasse 7, 53111 Bonn --- Germany\newline
	\href{mailto:flore@mpim-bonn.mpg.de}{\sf flore@mpim-bonn.mpg.de}
}

\usepackage{hyphenat}
\begin{document}
\normalem
\begin{abstract}
We record an explicit proof of the theorem that lifts a two\hyp{}variable adjunction to the arrow categories of its domains.
\end{abstract}
\allowdisplaybreaks
\maketitle
\section{Introduction}
The orthogonality relation between morphisms in a category certainly lies at the core of category theory; the reliance of higher category theory on factorization systems has been for a long time indisputable. 

As already stated in \cite{tderiv,FL0}, the notion of a \emph{factorization system} dates back to the very prehistory of category theory: Mac Lane's \cite{maclane1948groups} was able to devise what we would now call a factorization system on the category $\mathsf{Grp}$ of groups; it is nevertheless only in \cite{FK} that a perspicuous understanding of the definition was reached, together with a demonstration of its ubiquity. 

Just a few years before, algebraic topology was about to be reborn: Quillen's \cite{Qui} devised the notion of \emph{model category} to attribute a ``homotopy theory'' to many algebraic and combinatorial structures. The systematic theory of such structures is now called \emph{homotopical algebra}, and somehow realizes the homotopy theorists' dream (see the introduction of \cite{Baues1989}) to reach an axiomatic theory of categories whose syntax allows to speak about homotopy theory.

The entire note will be devoted to prove the following result:
\begin{lem}[\protect{\cite{Hirschhorn2003}}, 9.3.6]
Let $\clM$ be a simplicial model category. If $A\to B$ and $X\to Y$ are maps in $\clM$, and $L\to K$ is a map of simplicial sets, then the following conditions are equivalent:
\begin{itemize}
\item The dotted arrow exists in every solid arrow diagram of the form
\[
\centeredxymatrix{
L\ar[r] \ar[d]_u & \ar[d] \underline{\hom}(B,X)\\
K\ar[r] \ar@{.>}[ur] & \underline{\hom}(A,X)\underset{\underline{\hom}(B,Y)}\times\underline{\hom}(A,Y)
}\]
in the category $\textbf{sSet}$ of simplicial sets.
\item The dotted arrow exists in every solid arrow diagram of the form
\[
\centeredxymatrix{
A\ar[r]\ar[d]_f & \ar[d] X^K\\
B\ar[r] \ar@{.>}[ur] & X^L \times_{Y^L} Y^K
}\]
in the category $\clM$.
\item The dotted arrow exists in every solid arrow diagram of the form
\[
\centeredxymatrix{
A\otimes K \amalg_{A\otimes L} B\otimes L\ar[d]\ar[r] & X\ar[d] \\
B\otimes K \ar[r]\ar@{.>}[ur]& Y
}
\]
in the category $\clM$.
\end{itemize}
\end{lem}
 This result is stated as a direct consequence of \cite[Definition 9.1.6]{Hirschhorn2003}: a non-expert eye might believe they could get away quite easily with its proof.

This is only partly true. Altough there is nothing difficult in the proof other than plain old category theory, yet it is precisely this flavour of category theory that is needed while playing with model structures. In the present note we embark in a full solution of this exercise. 
\section{The lifting theorem}
\begin{df}[\textsc{thc} situation]\label{tiaccaci}
Let $\S,\A,\B$ be three categories; a \thc situation is a triple $\t$ of functors
\[
\tens : \A\times \S\to \B \qquad\qquad
\lbkt : \B^\text{op}\times \S \to \A\qquad\qquad
\rbkt : \A^\text{op}\times \B\to \S
\]
related by the following isomorphisms, natural in each component $A\in \A, B\in\B, S\in\S$:
\[
\B(A\tens S, B)\cong \A(A, \lbkt[B][S])\cong \S(S, \rbkt[A][B])
\]
\end{df}
\begin{oss}
In more modern literature (see again textbooks pn model categories, \cite{Hov,Hirschhorn2003} or \cite{Cheng2014}) \thc situations are called \emph{adjunctions of two variables}. The name \emph{\thc situation} dates back to \cite{Graya}, motivated by the following ``standard'' example of \thc situation.
\end{oss}
\begin{ex}[The standard \thc situation]
Let $\C \in\Vcat$ be a category enriched on the monoidal category $\clV$; a tensored and cotensored (see \cite{Kelly2005}) structure on $\C$ put this category in the \emph{standard} \thc situation, defined by the tensor $(V,C)\mapsto C\tens V$, cotensor $(V,C)\mapsto C^V$ and internal hom $(C,C')\mapsto \C(C,C')$ functors. 
\end{ex}
\begin{notat}
We will use round brackets to denote the co/wedge components of such a co/unit: so given a \thc situation $\t$ on $\A,\B,\S$, the unit $\eta_{A,(K)}: A\to \lbkt[K][A\tens K]$ of the adjunction $\firstblank\tens K\dashv \lbkt[K][\firstblank]$ is natural in $A$ and a wedge in $K\in\S$; dually, the counit $\varepsilon_{Y,(K)} : \lbkt[K][Y]\tens K\to Y$ of the same adjunction is natural in $Y$ and a cowedge in $K\in\S$, and similarly for all the other co/units and adjunctions.
\end{notat}
\begin{df}[Leibniz operations on arrows]\label{leibniz}
Let $\S,\A,\B$ be three categories, and $\t = \{\tens,\lbkt,\rbkt\}$ a \thc situation in the same notation of \autoref{tiaccaci}. Assume $\B$ has finite colimits, and $\A, \S$ admit finite limits.

We define the \emph{Leibniz operations} as the three functors
\[
\A^{[1]}\times \S^{[1]}\xto{\firstblank\odot\secondblank} \B^{[1]} \qquad
 (\B^{[1]})^\text{op}\times \S^{[1]} \xto{\llbkt} \A^{[1]}\qquad
 (\A^{[1]})^\text{op}\times \B^{[1]}\xto{\rrbkt} \S^{[1]}
\]
where given $f: A\to B$, $g: X\to Y$ and $u: L\to K$ in $\S$ the arrows $f\diamond u$, $\llbkt[u][g]$ and $\rrbkt[f][g]$ are defined, respectively, to be the dotted arrows in the following diagrams
\begin{center}
\includegraphics[width=\textwidth]{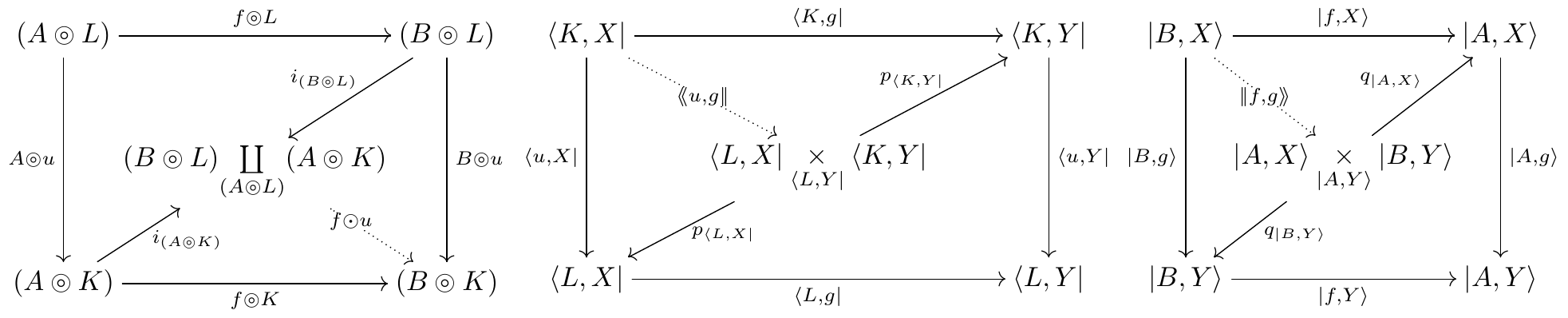}
\end{center}
\end{df}
\begin{thm}[The main theorem]\label{main}
Let $\S,\A,\B$ be three categories, and $\t$ a \thc situation. In the same notation of \autoref{tiaccaci}, the Leibniz operations of \autoref{leibniz} form another \thc situation (that we denote $\t^{[1]}$) on the functor categories $\S^{[1]},\A^{[1]},\B^{[1]}$; in other words, there are isomorphisms (natural in each component with respect to morphisms in the functor category)
\[
\B^{[1]}(f\diamond u, g)\overset{(\star)}\cong \A^{[1]}(f, \llbkt[u][g])\overset{(\star\star)}\cong \S^{[1]}(u, \rrbkt[f][g])
\]
induced by the functors $\firstblank\tens\secondblank, \lbkt, \rbkt$ of $\t$
\end{thm}
\begin{proof}
The proof of this statement occupies the rest of the section. As the reader will soon notice, the proof is elementary but it takes a certain effort to fill in all the details; we concentrate on the proof of isomorphism $(\star)$, conscious that the proof of $(\star\star)$ is similar in each part.
\begin{notat*}
It is of vital importance to establish good notation. We will then constantly refer to the diagrams above, defining the triple of functors $\{\diamond, \llbkt,\rrbkt\}$, without further mention, and we hereby choose the following shorthand to denote the co/limits appearing in those diagrams:
\[
P \coloneqq  (B\tens L) \coprod_{(A\tens L)} (A\tens K) ; \qquad
Q \coloneqq  \lbkt[L][X]\underset{\lbkt[L][Y]}\times \lbkt[K][Y]  ; \qquad
R \coloneqq \rbkt[A][X] \underset{\rbkt[A][Y]}\times \rbkt[B][Y]
\]
\end{notat*}
We start by defining the bijection $(\star)$; a generic morphism $(a,b): f\diamond u\to g$ in $\B$ is represented by a commutative square like
\[
\centeredxymatrix{
	P \ar[d]_{f\diamond u} \ar[r]^a & X \ar[d]^g \\
	B \tens K \ar[r]_b & Y
}
\]
so that a natural candidate for an arrow $\B^{[1]}(f\diamond u, g)\cong \A^{[1]}(f, \llbkt[u][g])$ is the following correspondence: the pair $(a,b)$ goes to the pair $(\hat a,\hat b)$ where 
\begin{itemize}
\item The arrow $\hat a$ is the \emph{mate} of the arrow $A\tens K \xto{i_{A\tens K}} P \xto{a} X$ under the adjunction $\firstblank\tens K \dashv \llbkt[K][\firstblank]$, i.e.
\[
\xymatrix@C=2cm{
	A \ar[r]_{\eta_{A,(K)}}& \llbkt[K][A\tens K] \ar[r]_{\llbkt[K][a\circ i_{A\tens K}]} & \llbkt[K][X]
}
\]
\item The arrow $\hat b$ is obtained via the universal property of $Q$, i.e\@. via the following diagram:
\[
\centeredxymatrix{
	B \ar@/_1pc/[ddr]_{b_{\lbkt[K][Y]}}\ar@/^1pc/[drr]^{b_{\lbkt[L][X]}}\ar@{.>}[dr]|{\hat b}& \\
	& Q \ar[r]^{p_{\lbkt[L][X]}}\ar[d]^{p_{\lbkt[K][Y]}}& \lbkt[L][X]\ar[d]^{\lbkt[L][g]} \\
	& \lbkt[K][Y] \ar[r]_{\lbkt[u][Y]} & \lbkt[L][Y]
}
\]
where $b_{\lbkt[L][X]}$ is $B \xto{\eta_{B, (L)}} (B\tens L)^L \xto{(a\circ i_{B\tens L})^L} \lbkt[L][X]$ and $b_{\lbkt[K][Y]}$ is the mate of $b$, $B \xto{\eta_{B, (K)}} (B\tens K)^K \xto{\lbkt[K][b]} \lbkt[K][Y]$.
\end{itemize}
Now we have to prove that this pair really defines a morphism $f\to \llbkt[u][g]$; to this end, we have to exploit the universal property of pullback, showing that the square
\[
\centeredxymatrix{
	A \ar[r]^{\hat a}\ar[d]_f & \lbkt[K][X] \ar[d]^{\llbkt[u][g]}\\
	B \ar[r]_{\hat b} & Q
}
\]
is commutative by pre-pending it with the projections $p_{\lbkt[L][X]}, p_{\lbkt[K][Y]}: Q\rightrightarrows \lbkt[L][X], \lbkt[K][Y]$ and showing that 
\[
\begin{cases}
p_{\lbkt[L][X]} \circ \llbkt[u][g] \circ \hat a = p_{\lbkt[L][X]} \circ \hat b\circ f \\
p_{\lbkt[K][Y]} \circ \llbkt[u][g] \circ \hat a = p_{\lbkt[K][Y]} \circ \hat b\circ f
\end{cases}
\]
The commutativity of the square will follow from the joint monicity of the pair $(p_{\lbkt[L][X]}, p_{\lbkt[K][Y]})$.

To prove this last statement, we use the chain of equations
\begin{align*} 
p_{\lbkt[L][X]} \circ \llbkt[u][g] \circ \hat a &= \underline{\lbkt[u][X] \circ \lbkt[K][a]}_{\red{1}}\circ \lbkt[K][i_{A\tens K}] \circ \eta_{A, (K)}\\
&= \lbkt[L][a] \circ\underline{\lbkt[u][P] \circ \lbkt[K][i_{A\tens K}]}_{\red{2}} \circ\eta_{A, (K)}\\
&=\lbkt[L][a] \circ \lbkt[L][i_{A\tens K}] \circ\underline{(A\tens K)^u \circ\eta_{A, (K)}}_{\red{3}}\\
&=\lbkt[L][a] \circ\underline{\lbkt[L][i_{A\tens K}] \circ \lbkt[L][A\tens u]}_{\red{4}} \circ\eta_{A,(L)}\\
&= \lbkt[L][a] \circ \lbkt[L][i_{B\tens L}] \circ\underline{\lbkt[L][f\tens L] \circ\eta_{A,(L)}}_{\red{5}}\\
&=\underline{\lbkt[L][a_1] \circ \eta_{B,(L)}} \circ f =b_{\lbkt[L][X]}\circ f = p_{\lbkt[L][X]}\circ \hat b\circ f; \\
p_{\lbkt[K][Y]}\circ \llbkt[u][g]\circ \hat a &= \lbkt[K][g] \circ \hat a\\
&= \lbkt[K][g] \circ \lbkt[K][a] \circ \lbkt[K][i_{A\tens K}] \circ \eta_{A,(K)}\\
&= \lbkt[K][b] \circ \underline{\lbkt[K][f\diamond u] \circ \lbkt[K][i_{A\tens K}]} \circ \eta_{A, (K)}\\
&= \lbkt[K][b] \circ \underline{\lbkt[K][f\tens K] \circ \eta_{A,(K)}}_{(5')}\\
&= \lbkt[K][b] \circ \eta_{B,(K)} \circ f\\
&= b_{\lbkt[K][Y]}\circ f = p_{\lbkt[K][Y]}\circ \hat b\circ f\\
\end{align*}
justified by the following commutative squares, all obtained from naturality and co/wedge conditions of the co/units of the starting \thc situation:
\begin{center}
\includegraphics[width=\textwidth]{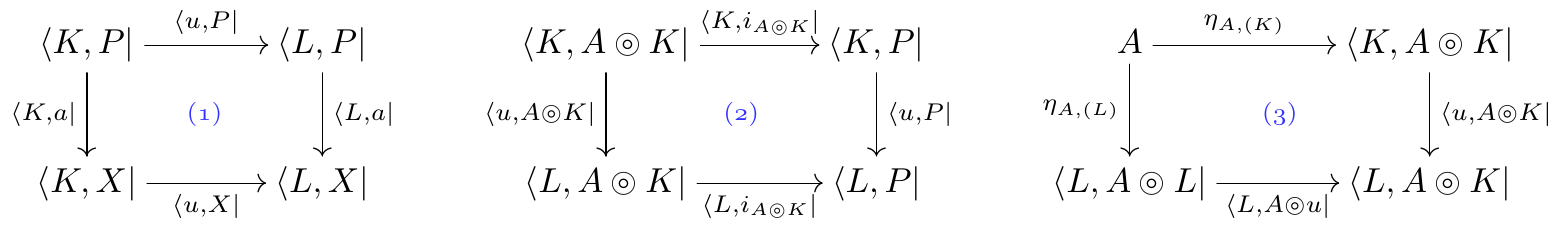}\\
\includegraphics[width=\textwidth]{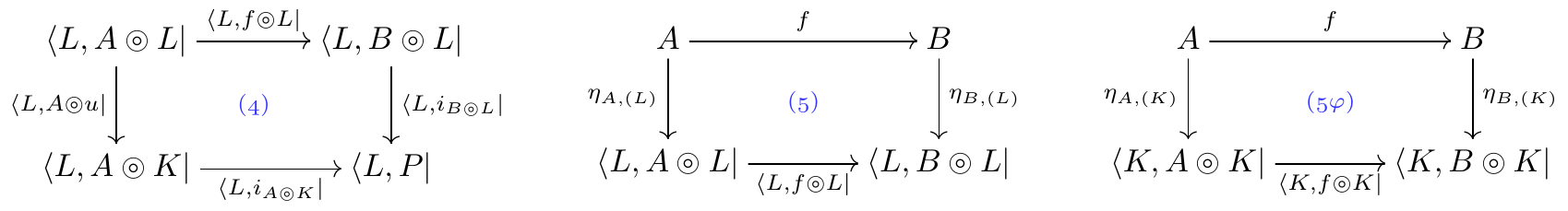}
\end{center}
This paves the way to a definition for an inverse of $\varphi$: let $(x,y): f\to \llbkt[u][g]$ be a morphism in $\A$, encoded in a commutative square
\[
\centeredxymatrix{
	A \ar[r]^-x\ar[d]_f & \lbkt[K][X] \ar[d]^{\llbkt[u][g]}\\
	B \ar[r]_-y & Q
}
\]
We define
\begin{itemize}
\item $\tilde x$ to be the arrow obtained via the universal property o $P$, from the diagram
\[
\centeredxymatrix{
	A\tens L \ar[r]^{A\tens u}\ar[d]_{f\tens L} & A\tens K \ar[d]_{i_{A\tens K}} \ar@/^1pc/[ddr]^{x_{A\tens K}} \\
	B\tens L  \ar@/_1pc/[drr]_{x_{B\tens L}} \ar[r]_{i_{B\tens L}} & P \ar@{.>}[dr]|{\tilde x} \\
	&& X
}
\]
where $(x_{A\tens K}, x_{B\tens L})$ is the pair of arrows 
\begin{align*}
x_{A\tens K} &= \epsilon_{X,(K)}\circ (x\tens K): A\tens K\to \lbkt[K][X]\tens K \to X\\
x_{B\tens L} &= \epsilon_{X, (L)}\circ p_{\lbkt[L][X]}\circ y\tens L: B\tens L \to \lbkt[L][X] \tens L \to X
\end{align*}
\item $\tilde y$ to be the mate of $p_{\lbkt[K][Y]}\circ y$ under the adjunction $\firstblank\tens K \dashv (-)^K$, i.e\@. the arrow $B\tens K \xto{p_{\lbkt[K][Y]}\circ y \tens K} \lbkt[K][Y]\tens K \xto{\epsilon_{Y, (K)}} Y$.
\end{itemize}
We show now that the pair $(\tilde x,\tilde y)$ defines a morphism $f\diamond u\to g$, i.e\@. that the square
\[
\centeredxymatrix{
	P \ar[d]_{f\diamond u} \ar[r]^{\tilde x} & X \ar[d]^g \\
	B \tens K \ar[r]_{\tilde y} & Y
}
\]
is commutative; the same reasoning above, suitably dualized, applies: we can exploit the joint epicity of the pushout inclusions $i_{A\tens K}: A\tens K\to P$ and $i_{B\tens L}: B\tens L\to P$ to prove that
\[
\begin{cases}
g \circ \tilde x \circ i_{B\tens L} = \tilde y \circ (f\diamond u)\circ i_{B\tens L}\\
g \circ \tilde x \circ i_{A\tens K} = \tilde y \circ (f\diamond u)\circ i_{A\tens K}.
\end{cases}
\]
Now, we use a similar chain of equations,
\begin{align*}
g\circ \tilde x \circ i_{B\tens L} &= \underline{g\circ \epsilon_{X, (L)}}_{\red{1}}\circ (p_{\lbkt[L][X]}\circ (y\tens L))\\
& = \epsilon_{Y,(L)}\circ (\underline{\lbkt[L][g] p_{\lbkt[L][X]}}_{\red{2}}\circ(y\tens L))\\
& = \underline{\epsilon_{Y,(L)}\circ (\lbkt[u][Y] \tens L)}_{\red{3}} \circ (p_{\lbkt[K][Y]}\circ(y\tens L))\\
& = \epsilon_{Y,(K)} \circ \underline{(\lbkt[K][Y]\tens u)\circ (p_{\lbkt[K][Y]}\circ(y\tens L))}_{\red{4}}\\
& = \underline{\epsilon_{Y,(K)} \circ (p_{\lbkt[K][Y]}\circ (y\tens K))}\circ \dotuline{(B\tens u)}\\
& = \underline{\tilde y}\circ \dotuline{(f\diamond u)\circ i_{B\tens L}}\\
g\circ \tilde x \circ i_{A\tens K} &= \underline{g\circ \epsilon_{X,(K)}}_{(1')}\circ (X\tens K)\\
 &= \epsilon_{Y,(K)}\circ \underline{(\lbkt[K][g]\tens K)} \circ (x\tens K) \\
 &= \epsilon_{Y,(K)}\circ (p_{\lbkt[K][Y]}\circ \underline{\llbkt[u][g])\tens K \circ (x\tens K)}\\
 &= \epsilon_{Y,(K)}\circ (p_{\lbkt[K][Y]}\circ (yf\tens K))\\
 &= \underline{\epsilon_{Y,(K)}\circ (p_{\lbkt[K][Y]}\circ (y\tens K))}\circ \dotuline{f\tens K}\\
 &= \underline{\tilde y}\circ \dotuline{(f\diamond u)\circ i_{A\tens K}}
\end{align*}
justified by the following commutative squares, all obtained from the definitions and naturality of the co/units of the starting \thc situation:
\begin{center}
\includegraphics[width=\textwidth]{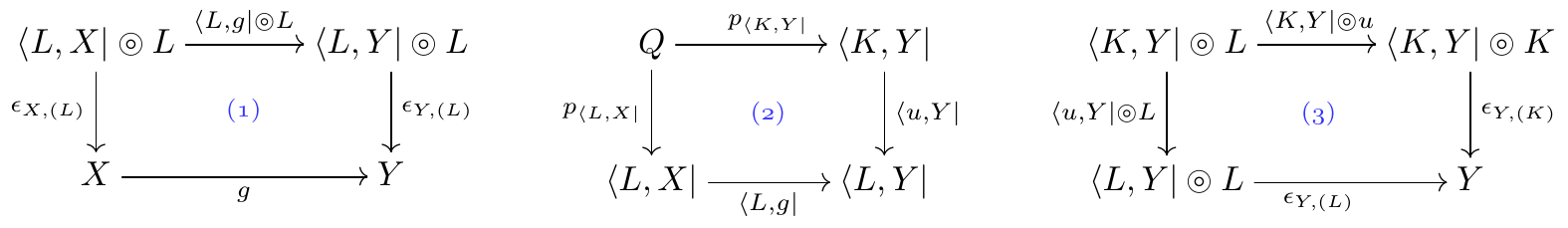}\\
\includegraphics{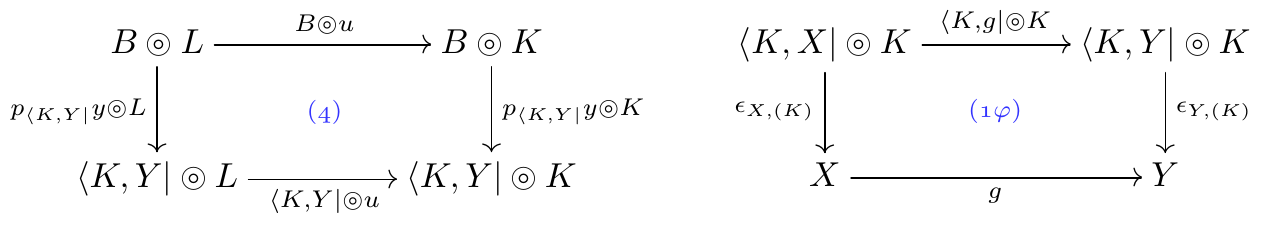}
\end{center}
It remains only to prove, now, that these two correspondences are mutually inverse: this exploits the triangle identities of the adjunctions in the \thc situation: in the above notation, starting from $(a,b): f\diamond u \to g$ we know that
\begin{itemize}
\item $\tilde{\hat a}$ is the arrow induced by the pair $(\epsilon_{X,(K)}\circ (\hat a\tens K),\epsilon_{X,(K)}\circ (p_{\lbkt[L][X]}\hat b\tens L))$: we can compute
\begin{align*}
\epsilon_{X,(K)}\circ (\hat a\tens K)   &= \epsilon_{X,(K)}\circ \big((\lbkt[K][a]\circ i_{A\tens K}^K) \tens K\big)\circ \eta_{A,(K)}\tens K \\
& = a \circ i_{A\tens K}\circ \underline{\epsilon_{A\tens K, (K)}\circ (\eta_{A, (K)}\tens K)}\\
& = a \circ i_{A\tens K};\\
\epsilon_{X,(K)}\circ (p_{\lbkt[L][X]}\hat b\tens L) &= \epsilon_{Y,(K)}\circ \big( (a\circ i_{B\tens L})^L\circ \eta_{B,(L)}\big)\tens L \\
& = a\circ i_{B\tens L}\circ \underline{\epsilon_{B\tens L, (L)}\circ \eta_{B,(L)}\tens L}\\
& = a\circ i_{B\tens L}.
\end{align*}
and we conclude, since $a$ is (tautologically) induced by the pair $(a \circ i_{A\tens K},a\circ i_{B\tens L})$.
\item $\tilde{\hat b}$ is the mate of $p_{\lbkt[K][Y]}\circ \hat b = \lbkt[K][b]\circ \eta_{B,(K)}$: but now
\begin{align*}
\tilde{\hat b} &= \epsilon_{Y, (K)}\circ (\lbkt[K][b]\circ \eta_{B,(K)})\tens K\\
&= b\circ \underline{\epsilon_{B\tens K, (K)}\circ (\eta_{B,(K)}\tens K)}\\
&= b.
\end{align*}
\end{itemize}
The other composition exploits in a similar way the triangle identities of an adjunction, and concludes the proof of the fact that 
\[
\B^{[1]}(f\diamond u, g) \cong \A^{[1]}(f, \llbkt[u][g]).
\]
A similar argument now shows the isomorphism
\[
\B^{[1]}(f\diamond u, g) \cong \A^{[1]}(u, \rrbkt[f][g]),
\]
and all these passages are natural in their arguments, thus showing the existence of the desired \thc situation.\footnote{Until now, we insisted in a certain pedantry because we find that this is useful to the novice, approaching the subject from outside pure category theory, or the one looking for a complete argument; now we feel free to relax, as the proof of the fact that $\B^{[1]}(f\diamond u, g) \cong \A^{[1]}(u, \rrbkt[f][g])$ would go exactly the same way, exploiting exactly the same technique, simply exchanging the r\^ole of $f$ and $u$, and the r\^ole of $\llbkt$ with $\rrbkt$.}
\end{proof}
It is worth to notice that the adjunction isomorphisms are compatible with the orthogonality relation, in the sense made precise by the following statement:
\begin{oss}
Let $\t$ be a \thc situation, and $\t^{[1]}$ the induced \thc following the notation of \autoref{main}. Then the bijections
\[
\B^{[1]}(f\diamond u, g) \cong \A^{[1]}(f, \llbkt[u][g])
\cong \A^{[1]}(u, \rrbkt[f][g])
\] (``bijections between lifting problems'') restrict to bijections between the squares that admit an orthogonal, or weakly orthogonal, lifting (``bijections between solutions of lifting problems''). More precisely, in the following commutative diagrams
\[
\centeredxymatrix{
P\ar[r]^{a}\ar[d]_{f\diamond g} & L\ar[d]^u\\
B\tens K\ar[r]_{b} \ar@{.>}[ur]_{\alpha}& K 
}
\qquad
\centeredxymatrix{
A\ar[r]^{\hat a}\ar[d]_f & \ar[d]^{\llbkt[u][g]}\lbkt[K][X]\\
B\ar[r]_{\hat b} \ar@{.>}[ur]|{\hat \alpha}&  Q
}
\qquad
\centeredxymatrix{
L\ar[r]^{\tilde a}\ar[d]_u & \ar[d]^{\rrbkt[f][g]}\rbkt[B][X]\\
K\ar[r]_{\tilde b} \ar@{.>}[ur]|{\tilde \alpha}& R
}\label{liftiproble} 
\] if a solution $\alpha$ exists in any one of the three lifting problems, splitting the square into two commutative triangles, then it exists in the other two, and this solution is precisely the \emph{mate} of $\alpha$ under the adjunction maps generated by $\t$. 
\end{oss}
At first glance, this is similar to, and generalizing, the fact that if $F\adjunct{\eta}{\epsilon} G$ is a pair of adjoint functors, then $F(f)\boxslash g$ in the domain of $G$ if and only if $f \boxslash G(g)$ in the domain of $F$: lifting problems exchange each other under the adjunction isomorphism, in the sense that the left lifting problem in the diagram below,
\[
\centeredxymatrix{
FA \ar[r]\ar[d]& X \ar[d]\\
FB \ar@{.>}[ur]_\alpha \ar[r]& Y
}\qquad\qquad
\centeredxymatrix{
A \ar[r]\ar[d]& GX \ar[d]\\
B \ar@{.>}[ur]_{\hat\alpha} \ar[r]& GY
}
\] 
can be solved by $\alpha : FB\to Y$ precisely if the lifting problem on the right can be solved by $\hat\alpha = G\alpha\circ \eta_B$. However, there is no chance to reduce ($\star$) above to this remark, as the morphism $f\diamond u$ is not the result of the action of a functor $\firstblank\diamond u$ on $f$ (it is, instead, the function on \emph{objects} of such a functor $\A^{[1]}\times \S^{[1]}\xto{\firstblank\odot u} \B^{[1]}$).

Hence, we are really left with the statement to prove, essentially from scratch; the argument will not be really different from what the reader can imagine, so we concede ourselves to be again slightly sketchy.

Consider the diagram
\[{\small
\begin{tikzcd}
A \arrow[ddd, "f"'] \arrow[r, "{\eta_{A,(K)}}"] & \langle K,A\tens K| \arrow[r, "{\langle K, i_{A\tens K}|}"] \arrow[dd, "{\langle K, f\tens K|}"'] & \lbkt[K][P] \arrow[r, "{\lbkt[K][a]}"] \arrow[d, "{\langle K, f\odot u|}"'] & \lbkt[K][X] \arrow[ddd] \arrow[rrrddd, "{\llbkt[u][g]}", bend left] &  &  &  \\
 &  & \langle K, B\tens K| \arrow[ru, "{\langle K,\alpha|}" description] &  &  &  &  \\
 & \langle K,B\tens K| \arrow[ru, equal, "{\text{id}_{\langle K,B\tens K|}}"'] &  &  &  &  &  \\
B \arrow[ru, "{\eta_{B,(K)}}"'] \arrow[rrr, "\hat b"] \arrow[d, "{\eta_{B,(K)}}"'] &  &  & Q \arrow[rrr, "{p_{\lbkt[L][X]}}"] \arrow[d, "{p_{\lbkt[K][Y]}}"] &  &  & \lbkt[L][X] \arrow[d, "{\langle L, g|}"] \\[1cm]
\langle K, B\tens K| \arrow[rrr, "{\lbkt[K][b]}"'] &  &  & \lbkt[K][Y] \arrow[rrr, "{\langle u, Y|}"'] &  &  & \lbkt[L][Y]
\end{tikzcd}}
\notag
\] 
the north-west upper triangle commutes because it can be split in three subdiagrams, each of which commutes, giving the chain of equalities
\begin{align*}
\hat\alpha \circ f &= \lbkt[K][\alpha]\circ \eta_{B,(K)}\circ f \\
&=\lbkt[K][\alpha]\circ \lbkt[K][f\tens K]\circ \eta_{A,(K)}\\
&=\lbkt[K][\alpha\circ (f\tens K)]\circ \eta_{A,(K)}\\
&=\lbkt[K][\alpha\circ (f\odot u)\circ i_{A\tens K}]\circ \eta_{A,(K)}\\
&=\lbkt[K][a\circ i_{A\tens K}]\circ \eta_{A,(K)}\\
&=\hat a.\\
\end{align*}
Now we have to prove that the north-west lower tirangle commutes, knowing that all other parts of the diagram commute. With a similar strategy as before, we prepend the arrows $\llbkt[u][g]\circ \lbkt[K][\alpha] \circ \eta_{B,(K)}$ and $\hat b$ with the pullback projections, and we consider the chain of equalities
\begin{align*}
p_{\lbkt[K][Y]} \circ \underline{\llbkt[u][g] \circ \lbkt[K][\alpha]}\circ \eta_{B,(K)} &= \lbkt[K][g] \circ \lbkt[K][\alpha] \circ \eta_{B,(K)}\\
&= \lbkt[K][y] \circ \eta_{B,(K)}\\
&= p_{\lbkt[K][Y]} \circ \hat b\\
p_{\lbkt[L][X]} \circ \llbkt[u][g]\circ \lbkt[K][\alpha] \circ \eta_{B,(K)} &= \underline{\lbkt[u][X] \circ \lbkt[K][\alpha]} \circ \eta_{B,(K)}\\
&= \lbkt[L][\alpha] \circ \underline{\lbkt[u][B\tens K] \circ \eta_{B,(K)}}\\
&= \underline{\lbkt[L][\alpha] \circ \lbkt[L][B\tens u]} \circ \eta_{B,(K)}\\
&= \lbkt[L][\alpha] \circ (f\odot u)^L \circ \underline{\lbkt[L][i_{B\tens L}]\circ \eta_{B,(K)}}\\
&= \lbkt[L][\alpha] \circ \lbkt[L][i_{B\tens L}]\circ \eta_{B,(K)}\\
&= p_{\lbkt[L][X]} \circ \hat b,
\end{align*}
motivated again by a mixture of naturality, dinaturality (e.g. of $\eta_{B,(K)}$), the definition of $\hat b$ above, and by the fact that $\alpha$ solves the lifting problem in \eqref{liftiproble}; the joint monicity of the pullback projections now entails that $\llbkt[u][g]\circ \lbkt[K][\alpha] \circ \eta_{B,(K)} = \hat b$, as desired. In a similar fashion it is possible to prove that $\tilde\alpha$ solves the appropriate lifting problem, and the converse implication.
\section{Applications}
Recall that a class of arrows $\clX\subseteq\bX^{[1]}$ in a category is \emph{saturated} if it is wide (i.e. it contains all isomorphisms and it is closed under composition), and closed under retracts, pushouts and transfinite composition. Every class of arrows $\clA$ \emph{generates} a saturated class $\clA^\text{s}$ as intersection of all saturated classes $\clX\supseteq \clA$. Dually we define a \emph{cosaturated} class $\clY$ to be wide, closed under retracts, pullbacks and transfinite op-composition. Every class of arrows $\clA$ generates, similarly, a cosaturated class $\prescript{\text{s}\!\!}{}{\clA}$ as intersection of all cosaturated classes $\clY\supseteq \clA$.
\begin{thm}
Let $\t$ be a \thc situation, and $\t^{[1]}$ the induced \thc following the notation of \autoref{main}. Let $\clA,\clS, \clB$ be respectively classes of objects in $\A,\S,\B$. 
\begin{itemize}
\item Let $(\clE_B,\clM_B)$ be a weak factorization system on $\B$. Then, if $\clA\odot \clS \subseteq \clE_B$, so does the class $\clA^\text{s}\odot \clS^\text{s}$ of $\odot$-products of the saturated closures of $\clA,\clS$.
\item Let $(\clE_A,\clM_A)$ be a weak factorization system on $\A$. Then, if $\llbkt[\clS][\clA] \subseteq \clM_A$, so does the class $\llbkt[\clS^\text{s}][\prescript{\text{s}\!}{}{\clB}]$.
\item Let $(\clE_S,\clM_S)$ be a weak factorization system on $\S$. Then, if $\rrbkt[\clA][\clB] \subseteq \clM_S$, so does the class $\rrbkt[\clA^\text{s}][\prescript{\text{s}\!}{}{\clB}]$. 
\end{itemize}
\end{thm}
\begin{proof}
We prove that $\clA\odot \clS \subseteq \clE_B$ implies $\clA^\text{s}\odot \clS \subseteq \clE_B$; the proof that $\clA\odot \clS^\text{s} \subseteq \clE_B$ is similar and will follow from the same argument (of course, there are many cases of interest --- like when $\firstblank\tens\secondblank$ defines a symmetric monoidal structure on $\A=\S=\B$ --- in which the proof simplifies and this further step is not necessary). Recall that $\clE_B = \clM_B^\boxslash$, so that it is enough to prove that $(\clA\odot \clS) \boxslash \clM_B$ entails $(\clA^\text{s} \odot \clS) \boxslash \clM_B$. To prove this, note that it is sufficient to show that 
\[
\forall s\in\clS, \forall m\in\clM_B\quad (\clA\odot s) \boxslash \clM_B.
\]
Define $\clK \coloneqq \{ u \mid u\odot s\boxslash m \quad \forall s\in\clS,\forall m\in\clM_B \} $: our main theorem entails that
\[
\clK = \{u\mid u\boxslash \llbkt[s][m] \quad \forall s\in\clS,\forall m\in\clM_B\} = \prescript{\boxslash}{}{\llbkt[\clS][\clM_B]}
\]
which entails that $\clK$ is a saturated class, so that if $\clK \supseteq \clA$, it follows that $\clK \supseteq \clA^\text{s}$. Similarly, we have the chain of equalities
\begin{align}
\clH &\coloneqq \{s\mid u\odot s \boxslash m \quad \forall u\in\clA,\forall m\in \clM_B\}\notag\\
&=\{s\mid s \boxslash \rrbkt[u][m] \quad \forall u\in\clA,\forall m\in \clM_B\}\notag\\
&= \prescript{\boxslash}{}{\rrbkt[\clA][\clM]}
\end{align}
so that $\clH$ is a saturated class, and if $\clK \supseteq \clA$, it follows that $\clK \supseteq \clA^\text{s}$. It then follows that
\[\clA \odot \clS \subseteq \clM_B \Rightarrow \clA^\text{s}\odot\clS \subseteq \clM_B \Rightarrow \clA^\text{s}\odot\clS^\text{s} \subseteq \clM_B.\qedhere\]
\end{proof}

\paragraph{\bf Acknowledgements.} This document was partially written during the author's stay at MPIM.

\bibliography{allofthem}{}
\bibliographystyle{amsalpha}
\hrulefill 
\end{document}